\batchmode
\documentclass[12pt,oneside,english]{amsart}
\usepackage[T1]{fontenc}
\usepackage[latin9]{inputenc}
\usepackage{babel}
\usepackage{verbatim}
\usepackage{amsthm}
\usepackage{amssymb}
\usepackage{graphicx}
\PassOptionsToPackage{normalem}{ulem}
\usepackage{ulem}
\usepackage[unicode=true,pdfusetitle,
 bookmarks=true,bookmarksnumbered=false,bookmarksopen=false,
 breaklinks=false,pdfborder={0 0 1},backref=false,colorlinks=false]
 {hyperref}
\usepackage{breakurl}

\makeatletter
\numberwithin{equation}{section}
\numberwithin{figure}{section}
\theoremstyle{plain}
\newtheorem{thm}{\protect\theoremname}[section]
  \theoremstyle{plain}
  \newtheorem{conjecture}[thm]{\protect\conjecturename}
  \theoremstyle{plain}
  \newtheorem{cor}[thm]{\protect\corollaryname}
  \theoremstyle{plain}
  \newtheorem{lem}[thm]{\protect\lemmaname}
  \theoremstyle{remark}
  \newtheorem{rem}[thm]{\protect\remarkname}
  \theoremstyle{definition}
  \newtheorem{defn}[thm]{\protect\definitionname}
  \theoremstyle{definition}
  \newtheorem{example}[thm]{\protect\examplename}

\usepackage{ae}
\usepackage{aecompl}

\makeatother

  \providecommand{\conjecturename}{Conjecture}
  \providecommand{\corollaryname}{Corollary}
  \providecommand{\definitionname}{Definition}
  \providecommand{\examplename}{Example}
  \providecommand{\lemmaname}{Lemma}
  \providecommand{\remarkname}{Remark}
\providecommand{\theoremname}{Theorem}

\begin{document}
\begin{comment}
Group theory
\end{comment}

\global\long\def\normalin{\mathrel{\lhd}}

\global\long\def\innormal{\mathrel{\rhd}}

\global\long\def\semidirect{\mathbin{\rtimes}}

\global\long\def\Stab{\operatorname{Stab}}

\begin{comment}
Topology
\end{comment}

\global\long\def\bdry{\partial}

\global\long\def\susp{\operatorname{susp}}

\begin{comment}
Poset combinatorics
\end{comment}

\global\long\def\lrprod{\mathop{\check{\prod}}}

\global\long\def\lrtimes{\mathbin{\check{\times}}}

\global\long\def\urtimes{\mathbin{\hat{\times}}}

\global\long\def\urprod{\mathop{\hat{\prod}}}

\global\long\def\subsetdot{\mathrel{\subset\!\!\!\!{\cdot}\,}}

\global\long\def\dotsupset{\mathrel{\supset\!\!\!\!\!\cdot\,\,}}

\global\long\def\precdot{\mathrel{\prec\!\!\!\cdot\,}}

\global\long\def\dotsucc{\mathrel{\cdot\!\!\!\succ}}

\global\long\def\des{\operatorname{des}}

\global\long\def\rank{\operatorname{rank}}

\global\long\def\height{\operatorname{height}}

\begin{comment}
Lattice
\end{comment}

\global\long\def\modreln{\mathrel{M}}

\begin{comment}
Simplicial complex combinatorics
\end{comment}

\global\long\def\link{\operatorname{link}}

\global\long\def\freejoin{\mathbin{\circledast}}

\global\long\def\stellarsd{\operatorname{stellar}}

\global\long\def\conv{\operatorname{conv}}

\global\long\def\disjointunion{\mathbin{\dot{\cup}}}

\global\long\def\skel{\operatorname{skel}}

\global\long\def\depth{\operatorname{depth}}

\global\long\def\st{\operatorname{star}}

\global\long\def\alexdual#1{#1^{\vee}}

\global\long\def\reg{\operatorname{reg}}

\global\long\def\shift{\operatorname{Shift}}

\begin{comment}
Graph complexes
\end{comment}

\global\long\def\Dom{\operatorname{Dom}}

\begin{comment}
My stuff
\end{comment}

\global\long\def\cosetposet{\overline{\mathfrak{C}}}

\global\long\def\cosetlat{\mathfrak{C}}

\global\long\def\length{\operatorname{length}}

\global\long\def\lexlt{<_{\mathrm{lex}}}

\global\long\def\lexleq{\leq_{\mathrm{lex}}}

\subjclass[2010]{Primary 05E45; Secondary 20D30, 06A07. }

\title{Chains of modular elements and shellability}

\author{Russ Woodroofe}

\address{Department of Mathematics, Washington University in St.~Louis, St.~Louis,
MO, 63130}

\email{russw@math.wustl.edu}
\begin{abstract}
Let $L$ be a lattice admitting a left-modular chain of length $r$,
not necessarily maximal. We show that if either $L$ is graded or
the chain is modular, then the $(r-2)$-skeleton of $L$ is vertex-decomposable
(hence shellable). This proves a conjecture of Hersh. Under certain
circumstances, we can find shellings of higher skeleta. For instance,
if the left-modular chain consists of every other element of some
maximum length chain, then $L$ itself is shellable. We apply these
results to give a new characterization of finite solvable groups in
terms of the topology of subgroup lattices.

Our main tool relaxes the conditions for an $EL$-labeling, allowing
multiple ascending chains as long as they are lexicographically before
non-ascending chains. We extend results from the theory of $EL$-shellable
posets to such labelings. The shellability of certain skeleta is one
such result. Another is that a poset with such a labeling is homotopy
equivalent (by discrete Morse theory) to a cell complex with cells
in correspondence to weakly descending chains. 
\end{abstract}
\maketitle

\section{\label{sec:Introduction}Introduction}

We consider the order complex of a lattice admitting a chain $\mathbf{m}$
consisting of modular elements. The case where $\mathbf{m}$ is a
maximal chain has been studied systematically since \cite{Stanley:1972}:
such lattices are supersolvable. Supersolvable lattices were one motivation
for Björner's original definition of $EL$-labeling \cite{Bjorner:1980},
and in particular their order complexes are shellable and hence highly
connected.

Lattices that admit a non-maximal chain consisting of modular elements
are less well-understood. Hersh and Shareshian \cite{Hersh/Shareshian:2006}
used the Homotopy Complementation Formula to show that if $L$ has
a chain of length $r$ consisting of modular elements, then $L$ is
$(r-3)$-connected. The purpose of this paper is to extend Björner's
shellability results to situations of this type.

One motivation is to prove the following conjecture of Hersh, which
gives a new proof of the Hersh-Shareshian connectivity result:
\begin{conjecture}
\emph{\label{con:HershConj}(Hersh {[}personal communication{]})}
If $L$ is a finite lattice admitting a chain of length $r$ that
consists of modular elements, then the $(r-2)$-skeleton of $L$ is
pure and shellable.
\end{conjecture}
We prove the stronger result that the $(r-2)$-skeleton of $L$ is
vertex-decomposable. Moreover, if the $(r-1)$-skeleton of $L$ is
pure, then this is also vertex-decomposable, hence shellable. We also
show that we can weaken from modularity to left-modularity, provided
that the lattice is graded. More precise results are in Section \ref{sec:Shellings},
specifically Theorems \ref{thm:QuasiELShellsSkeleton}, \ref{thm:GradedQLeftModularIsVD},
and \ref{thm:ModRskel}.\medskip{}

Another motivation for study of lattices with a chain of modular elements
comes from an well-studied class of examples, that of a chain of normal
subgroups in a subgroup lattice $L(G)$. Combining Theorem \ref{thm:QuasiELShellsSkeleton}
with the results of \cite{Shareshian/Woodroofe:2012} gives:
\begin{thm}
\label{thm:Intro-SolvableShellSkel} If $G$ is a finite group with
a chief series of length $r$, then $G$ is non-solvable if and only
if the $(r-1)$-skeleton of $L(G)$ is shellable and pure of dimension
$(r-1)$. 
\end{thm}
An immediate consequence of Conjecture \ref{con:HershConj} is that
the $(r-2)$-skeleton of such an $L$ is Cohen-Macaulay, i.e. that
the depth of the simplicial complex is at least $r-2$. Depth is a
topological invariant, giving a new characterization of solvability
with respect to the topology of $L(G)$ and the length of a chief
series:
\begin{cor}
\label{cor:Intro-SolvableDepth} If $G$ is a finite group with a
chief series of length $r$, then $G$ is solvable if and only if
$\depth\vert L(G)\vert\leq r-2$.
\end{cor}
Corollary \ref{cor:Intro-SolvableDepth} is not the first topological
characterization of solvability, or even the first to involve shellability,
but it seems to have a quite different form from previous characterizations. 

\smallskip{}

The main tool used to show shellability of skeleta of posets will
be a certain relaxation of $EL$-labelings (and more generally of
$CL$-labelings). Our definition allows multiple ascending chains,
which are required to lexicographically precede all non-ascending
chains. In addition to shellability, we extend the theory of Björner
and Wachs \cite{Bjorner/Wachs:1996} to describe the homotopy type
of a lattice with such a labeling. Such labelings may have wider applicability
in proving depth bounds in other classes of lattices. Depth bounds
have interesting combinatorial consequences, including bounds on the
$f$-triangle \cite{Duval:1996}, as well as certain Erd\H{o}s-Ko-Rado
type results \cite{Woodroofe:2011a}. 

\medskip{}
The remainder of the paper is organized as follows. In Section \ref{sec:Background}
we review the necessary background on modularity, poset topology,
and shellability. In Section \ref{sec:QuasiLexLabelings} we extend
the definition of $CL$-labeling to that of a quasi-$CL$-labeling.
In Section \ref{sec:Shellings} we give shellings of skeleta in certain
dimensions of posets with a quasi-$CL$-labeling. We give particular
attention to applications in lattices possessing chains consisting
of (left\nobreakdash-)modular elements. In Section \ref{sec:DMtechniques},
we show how discrete Morse theory applies especially easily to posets
with a quasi-$CL$-labeling. In Section \ref{sec:SubgroupLatApplications}
we apply results of the preceding sections to the subgroup lattice
of a finite group.

All lattices, posets, simplicial complexes, and groups considered
in this paper are finite.\vspace{-0.3cm}

\section*{Acknowledgements}

I thank Patricia Hersh for bringing her conjecture to my attention.
I enjoyed several stimulating conversations with Hugh Thomas on modular
chains. I have benefited greatly from the interest and encouragement
of John Shareshian: his comments on the subgroup lattice aspects were
especially helpful. The anonymous referees gave detailed and helpful
comments, from which the paper has benefitted greatly.

\section{\label{sec:Background}Notation and background}

We assume general familiarity with poset topology and shellings as
found in e.g. \cite{Wachs:2007} and/or \cite{Jonsson:2008}, but
review the specific definitions and tools we will need.

\subsection{Modular and left-modular elements}

A pair $(x,y)$ from a lattice $L$ is a \emph{modular pair} if for
every $z\geq y$ we have that 
\[
(y\vee x)\wedge z=y\vee(x\wedge z).
\]
An element $x$ is \emph{left-modular} if $(x,y)$ is a modular pair
for every $y\in L$, and is \emph{modular} (or \emph{two-sided modular})
if both $(x,y)$ and $(y,x)$ are modular pairs for every $y\in L$.
We notice that left-modularity of $x$ is preserved in the lattice
dual $L^{*}$, but recall that (two-sided) modularity is not preserved.
The elements $\hat{0}$ and $\hat{1}$ of any lattice are easily seen
to be modular. We refer the reader to \cite{Birkhoff:1967} for additional
background on modularity, and to \cite{Liu/Sagan:2000} on left-modularity.

A \emph{(left-)modular chain} will refer to a chain consisting of
(left-)modular elements. A lattice is \emph{supersolvable} if it is
graded and has a left-modular maximal chain.

\subsection{Posets and topology}

Associated with any bounded partially-ordered set (\emph{poset}) $P$
is a simplicial complex $\vert P\vert$ (the \emph{order complex})\emph{
}with faces consisting of the chains of $P\setminus\{\hat{0},\hat{1}\}$.
When we say that $P$ satisfies some geometric property such as `shellable'
or `connected', we mean that $\vert P\vert$ satisfies the given property.

\subsection{Shellings}

A \emph{shelling} of a simplicial complex $\Delta$ is an ordering
$\sigma_{1},\dots,\sigma_{m}$ of the facets (maximal faces) of $\Delta$
such that the intersection of $\sigma_{i}$ with the subcomplex generated
by $\sigma_{1},\dots,\sigma_{i-1}$ is pure $(\dim\sigma_{i}-1)$-dimensional.
A useful equivalent characterization of a shelling order is that if
$i<k$, then there is a $j<k$ so that $\sigma_{i}\cap\sigma_{k}\subseteq\sigma_{j}\cap\sigma_{k}$
and $\vert\sigma_{j}\cap\sigma_{k}\vert=\vert\sigma_{k}\vert-1$.
A complex for which there exists a shelling is called \emph{shellable}. 

Any shellable complex is homotopy equivalent to a bouquet of spheres,
where the spheres correspond to (and have the same dimension as) certain
facets in the shelling. Every link in a shellable complex is also
shellable.

\subsection{Cohen-Macaulay, skeleta and depth}

We recall that a complex is \emph{Cohen-Macaulay over $k$} if $\tilde{H}_{i}(\link_{\Delta}\sigma;k)=0$
for all faces $\sigma$ (including $\sigma=\emptyset$) and $i<\dim(\link_{\Delta}\sigma)$.
Cohen-Macaulay complexes have interesting enumerative \cite{Duval:1996}
and extremal \cite{Woodroofe:2011a} properties, and are also of interest
via a connection to commutative algebra \cite{Stanley:1996} via the
Stanley-Reisner ring. One reason for study of shellable complexes
is that any \emph{pure} (having all facets of the same dimension)
shellable complex is Cohen-Macaulay. More generally, any shellable
complex is ``sequentially Cohen-Macaulay''. Additional background
on these properties can be found in e.g. \cite{Stanley:1996} or \cite{Wachs:2007}.

The \emph{$r$-skeleton} of a simplicial complex $\Delta$, which
we write as $\skel_{r}\Delta$, consists of all faces of dimension
$\leq r$. The \emph{depth} of a simplicial complex $\Delta$ is the
maximal $r\leq\dim\Delta$ such that $\skel_{r}\Delta$ is Cohen-Macaulay.
As with the Cohen-Macaulay property, $\depth\Delta$ is closely connected
to the depth (in the commutative algebra sense) of the associated
Stanley-Reisner ring. Moreover, the depth can be defined as a purely
topological property not depending on the triangulation of the underlying
space.

As stated in the introduction, our goal will be to construct shellings
of various skeleta of order complexes of posets. Let $m$ be the minimum
dimension of a facet of $\Delta$. Since any Cohen-Macaulay complex
is pure, we have that $\depth\Delta\leq m$. On the other hand, if
$\skel_{r}\Delta$ is shellable for some $r\leq m$, then $\depth\Delta\geq r$.

\subsection{Vertex-decomposability and $k$-decomposability}

We will use the following tool to construct most of the shellings
in this paper. A \emph{shedding vertex} of a simplicial complex $\Delta$
is a vertex $v$ such that for any face $\sigma$ of $\Delta$ with
$v\in\sigma$, there is a vertex $w\notin\sigma$ which can be exchanged
for $v$, i.e. such that $(\sigma\setminus v)\cup w$ is a face of
$\Delta$. If $\Delta$ has a shedding vertex $v$ such that both
$\Delta\setminus v$ and $\link_{\Delta}v$ are shellable, then $\Delta$
is also shellable \cite[Lemma 6]{Wachs:1999b}. We recursively define
$\Delta$ to be \emph{vertex-decomposable} if $\Delta$ either is
a simplex or else has a shedding vertex $v$ such that both $\Delta\setminus v$
and $\link_{\Delta}v$ are vertex-decomposable. It follows immediately
that a vertex-decomposable complex is shellable. 

We will make frequent use of the following lemma:
\begin{lem}
\label{lem:JoinOfVDSkelsIsVD}\cite[Lemma 6.12]{Jonsson:2008} If
$\Sigma$ and $\Gamma$ are simplicial complexes such that $\skel_{r}\Sigma$
and $\skel_{s}\Gamma$ are pure and vertex-decomposable, then $\skel_{r+s+1}\Sigma*\Gamma$
is vertex-decomposable.
\end{lem}
We will prove a generalization of Lemma \ref{lem:JoinOfVDSkelsIsVD}.
A \emph{shedding face} is a face $\tau$ such that for any face $\sigma$
containing $\tau$ and any vertex $v\in\tau$, there is a vertex $w\notin\sigma$
such that $(\sigma\setminus v)\cup w$ is a face \cite{Jonsson:2005}.
A complex is recursively defined to be \emph{$k$-decomposable} if
either $\Delta$ is a simplex, or else has a shedding face $\tau$
with $\dim\tau\leq k$ such that both $\Delta\setminus\tau$ and $\link_{\Delta}\tau$
are $k$-decomposable. Thus vertex-decomposability is exactly $0$-decomposability.
Any $k$-decomposable complex is shellable; conversely, any shellable
$d$-dimensional complex is $d$-decomposable \cite{Provan/Billera:1980,Woodroofe:2011b}. 
\begin{lem}
\label{lem:JoinOfkDSkelsIskD}If $\Sigma$ and $\Gamma$ are simplicial
complexes such that $\skel_{r}\Sigma$ and $\skel_{s}\Gamma$ are
pure and $k$-decomposable, then $\skel_{r+s+1}\Sigma*\Gamma$ is
$k$-decomposable.\end{lem}
\begin{proof}
Let $\tau$ be a shedding face of $\skel_{r}\Sigma$ (the case where
$\tau$ is a shedding face of $\skel_{s}\Gamma$ is symmetric). Let
$\sigma\cup\gamma$ be a face of $\skel_{r+s+1}\Sigma*\Gamma$, where
$\sigma$ is a face of $\Sigma$ containing $\tau$ and $\gamma$
is a face of $\Gamma$. If $\dim\sigma>r$, then by purity of $\skel_{s}\Gamma$
we get that $\gamma$ can be extended by some vertex $w$ of $\Gamma$
to a larger face in $\Gamma$. If $\dim\sigma\leq r$, then by the
shedding face condition there is for any $v\in\tau$ a vertex $w$
of $\Sigma$ with $(\sigma\setminus v)\cup w$ a face of $\skel_{r}\Sigma$.
In either case, for any $v\in\tau$ we can produce a $w$ so that
$((\sigma\cup\gamma)\setminus v)\cup w$ is a face of $\skel_{r+s+1}\Sigma*\Gamma$,
hence $\tau$ is a shedding face in $\skel_{r+s+1}\Sigma*\Gamma$. 

We conclude the proof by remarking that $(\skel_{r}\Sigma)\setminus\tau=\skel_{r}(\Sigma\setminus\tau)$
is pure by the shedding vertex condition, and observing that 
\begin{align*}
(\skel_{r+s+1}\Sigma*\Gamma)\setminus\tau & =\skel_{r+s+1}(\Sigma\setminus\tau)*\Gamma,\mbox{ and}\\
\link_{\left(\skel_{r+s+1}\Sigma*\Gamma\right)}\tau & =\skel_{r+s-\dim\tau}\left(\link_{\Sigma}\tau*\Gamma\right).
\end{align*}
The result then follows by induction. \end{proof}
\begin{cor}
If $\Sigma$ and $\Gamma$ are simplicial complexes such that $\skel_{r}\Sigma$
and $\skel_{s}\Gamma$ are pure and shellable, then $\skel_{r+s+1}\Sigma*\Gamma$
is shellable, and indeed $\max\{r,s\}$-decomposable.\end{cor}
\begin{rem}
An analogue to Lemmas \ref{lem:JoinOfVDSkelsIsVD} and \ref{lem:JoinOfkDSkelsIskD}
for the Cohen-Macaulay property (i.e. for depth) can be proved via
the Künneth formula \cite[Lemma 2.12]{Woodroofe:2011a}. 
\end{rem}

\subsection{Edge labelings}

Studying the behavior of certain labelings of a poset often gives
information about the poset's combinatorics and topology. An \emph{edge
labeling} of $P$ is any map from the cover relations of $P$ to an
ordered label set. Each maximal chain then has an associated label
sequence (reading the labels from bottom to top of the chain), and
we order maximal chains lexicographically according to their label
sequences. An \emph{$EL$-labeling} is an edge labeling such that
every interval has a unique weakly ascending maximal chain, and this
ascending chain is first according to the lexicographic order on maximal
chains in this interval. It is well-known \cite{Bjorner:1980,Bjorner/Wachs:1996}
that a bounded poset with an $EL$-labeling is shellable. 
\begin{rem}
As Wachs discusses in \cite[Remark 3.2.5]{Wachs:2007}, one can alternatively
define $EL$-labelings to have a unique \uline{strictly} ascending
chain on every interval. For our present purposes, it is more helpful
to keep in mind the weakly ascending version.
\end{rem}
A frequently useful extension of the definition of an $EL$-labeling
is as follows. A \emph{rooted cover relation} is a cover relation
$x\lessdot y$ together with a maximal chain $\mathbf{r}$ from $\hat{0}$
to $x$. A \emph{chain-edge labeling} of $P$ is a map from the rooted
cover relations of $P$ to an ordered label set. A \emph{$CL$-labeling}
is a chain-edge labeling obeying similar conditions as for an $EL$-labeling:
i.e., such that every rooted interval has a unique (weakly) ascending
maximal chain, and this ascending chain is first according to the
lexicographic order on all maximal chains in this rooted interval.
A bounded poset with a $CL$-labeling is shellable, and more generally
many of the other useful properties of posets with an $EL$-labeling
may be generalized to posets with a $CL$-labeling \cite{Bjorner/Wachs:1996,Bjorner/Wachs:1997}.

Due to the usefulness of $EL$/$CL$-labelings in constructing shellings,
we sometimes call a poset with such a labeling \emph{$EL$-shellable}
or \emph{$CL$-shellable}.

\section{\label{sec:QuasiLexLabelings}Quasi-$CL$-labelings}

If $y\lessdot z$ is any cover relation and $x$ is left-modular,
then $y\vee x\wedge z$ is either $y$ or $z$. Moreover, if $y\vee x\wedge z=z$
then $\left(y\vee w\right)\wedge z=z$ for any $w>x$, and similarly
for $w<x$ we have $\left(y\vee w\right)\wedge z=y$ if $y\vee x\wedge z=y$.
Henceforth, let 
\[
\mathbf{m}=\left\{ \hat{0}=m_{0}<m_{1}<\dots<m_{r}=\hat{1}\right\} 
\]
 be a (not necessarily maximal) left-modular chain. We see that cover
relations admit a labeling 
\begin{alignat*}{2}
\lambda(y\lessdot z) & = & \,\max & \left\{ i\,:\, y\vee m_{i-1}\wedge z=y\right\} \\
 & = & \min & \left\{ i\,:\, y\vee m_{i}\wedge z=z\right\} .
\end{alignat*}
We will refer to $\lambda$ as the \emph{left-modular labeling} of
$L$ with respect to $\mathbf{m}$. In the case where $\mathbf{m}$
is a maximal chain, $\lambda$ is an $EL$-labeling \cite{Bjorner:1980,Liu:1999}.
\begin{defn}
\label{def:QuasiCL} A \emph{quasi-$CL$-labeling} will be a chain-edge
labeling of a poset $P$ such that on any interval $[x,y]$ with root
$\mathbf{r}$ we have:
\begin{enumerate}
\item Every (weakly) ascending maximal chain is a refinement of a specific
chain 
\[
\mathbf{a}^{\mathbf{r},[x,y]}=\left\{ x=a_{0}^{\mathbf{r},[x,y]}<a_{1}^{\mathbf{r},[x,y]}<\dots<a_{k}^{\mathbf{r},[x,y]}=y\right\} ,\mbox{ that}
\]
 
\item all cover relations on the interval $[a_{i-1}^{\mathbf{r},[x,y]},a_{i}^{\mathbf{r},[x,y]}]$
receive the same label $\alpha_{i}$, and that
\item the maximal extensions of $\mathbf{a}^{\mathbf{r},[x,y]}$ are (strictly)
lexicographically earlier than the other maximal chains on $[x,y]$
with root $\mathbf{r}$. 
\end{enumerate}
\end{defn}
As usual, edge labelings are special cases of chain-edge labelings,
and when $\lambda$ is an edge labeling obeying the above properties
we will call it a \emph{quasi-$EL$-labeling}.

It is immediate that any $CL$-labeling is a quasi-$CL$-labeling,
and that any quasi-$CL$-labeling induces a quasi-$CL$-labeling when
restricted to any rooted interval $[x,y]$. An example of a quasi-$EL$-labeling
that is not an $EL$-labeling is given in Figure \ref{fig:quasiELexample},
where the solid lines represent an edge labeled 1, and the dotted
lines represent an edge labeled 2. We remark that the pictured poset
is a lattice, that the element labeled $m$ is modular, and that the
pictured labeling is the modular labeling with respect to $\hat{0}<m<\hat{1}$. 

\begin{figure}
\includegraphics{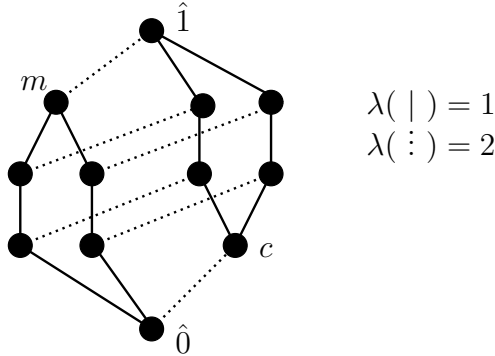}

\caption{A lattice with a quasi-$EL$-labeling.\label{fig:quasiELexample}}
\end{figure}

\smallskip{}

The following lemma is essentially \cite[Lemma 2.4]{Woodroofe:2008}.
For completeness we sketch the proof here.
\begin{lem}
\cite[Lemma 2.4]{Woodroofe:2008} If $\lambda$ is the left-modular
labeling with respect to $\mathbf{m}$, then $\lambda$ is a quasi-$EL$-labeling
with $\mathbf{a}^{[\hat{0},\hat{1}]}=\mathbf{m}$.\end{lem}
\begin{proof}
[Sketch.] Let $\mathbf{a}^{[x,y]}$ be the chain consisting of $\{x\vee m_{i}\wedge y\,:\,0\leq i\leq r\}$.
It is easy to see that every maximal extension of $\mathbf{a}^{[x,y]}$
is weakly ascending, that such chains are lexicographically earlier
than all other chains, and that every cover relation on $[x\vee m_{i-1}\wedge y,x\vee m_{i}\wedge y]$
receives label $i$.

Conversely, if $m_{k}$ is the least element of $\mathbf{m}$ with
$x\lneq x\vee m_{k}\wedge y$, and $w$ is an atom of the interval
$[x,y]$ with $w\nleq x\vee m_{k}$ (i.e., $w$ not on an extension
of $\mathbf{a}^{[x,y]}$), then $w<w\vee m_{k}\wedge y$. By minimality
of $m_{k}$, any maximal chain on $[x,y]$ that begins with $x\lessdot w$
contains a descent.
\end{proof}
The left-modular labeling was used in \cite{Woodroofe:2008} only
as a starting point to be refined to an $EL$-labeling of the subgroup
lattice of a finite solvable group. We notice that any quasi-$EL$-labeling
$\lambda_{q}$ of a bounded poset with an $EL$-labeling $\lambda_{r}$
can be refined to an $EL$-labeling by taking the new labeling $\lambda=(\lambda_{q},\lambda_{r})$,
where the labels are ordered lexicographically. Similarly for quasi-$CL$-labelings
and $CL$-labelings.
\begin{example}
Let $a_{1},\dots,a_{n}$ be any ordering of the atoms of a geometric
lattice $L$. It is well-known \cite[Section 3.2.3]{Wachs:2007} that
$\lambda_{*}(x\lessdot y)=\min\{i\,:\, a_{i}\vee x=y\}$ is an $EL$-labeling
of $L$. We notice that this $\lambda_{*}$ can be viewed as a refinement
of the modular quasi-$EL$-labeling $\lambda_{q}$ with respect to
the chain $\hat{0}<a_{1}<\hat{1}$, in the sense that it has the same
ascents and descents as $(\lambda_{q},\lambda_{*})$.
\end{example}

\section{\label{sec:Shellings}Shellings and vertex-decomposability}

We now extend the proof of Björner and Wachs \cite{Bjorner/Wachs:1997}
that a bounded poset with a $CL$-labeling is vertex-decomposable.
We first notice:
\begin{lem}
\label{lem:quasiELfirstdiff}If $\lambda$ is a quasi-$CL$-labeling
with $x$ an atom on an ascending chain of $[\hat{0},\hat{1}]$, then
$\lambda(\hat{0},x)<\lambda(\hat{0},y)$ for any atom $y$ not on
any ascending chain.
\end{lem}
\noindent The proof of Lemma \ref{lem:quasiELfirstdiff} is exactly
similar to that of \cite[Proposition 2.5]{Bjorner:1980}, and is omitted. 

We now state a technical lemma, paralleling \cite[Lemma 11.5]{Bjorner/Wachs:1997}. 
\begin{lem}
\label{lem:qELTech}Let $P$ be a bounded poset with a quasi-$CL$-labeling
$\lambda$, and let $x$ be the descent of a lexicographically greatest
member $\mathbf{c}$ of the collection of maximal chains with a single
descent. Then
\begin{enumerate}
\item \label{enu:qELTech-TrivInterval}Every chain on $[\hat{0},x]$ is
ascending, hence an extension of $\mathbf{a}^{[\hat{0},x]}$.
\item \label{enu:qELTech-NoAsc}No maximal chain $\mathbf{d}$ with $x\in\mathbf{d}$
has an ascent at $x$.
\item \label{enu:qELTech-KeepsCovers}If $w\lessdot x\lessdot z$, then
there is a $y\neq x$ such that $w<y<z$.
\item \label{enu:qELTech-Restricts}$\lambda$ restricts to a quasi-$CL$-labeling
of the induced subposet $P\setminus x$.
\end{enumerate}
\end{lem}
\begin{proof}
Let $\mathbf{r}$ be the restriction of $\mathbf{c}$ to $[\hat{0},x]$. 

(1.) We notice that $\mathbf{r}$ is a lexicographically greatest
member among all maximal chains of $[\hat{0},x]$, as otherwise a
lexicographically greater chain $\mathbf{r}'$ on $[\hat{0},x]$ together
with a maximal extension of $\mathbf{a}^{\mathbf{r}',[x,\hat{1}]}$
would be lexicographically greater than $\mathbf{c}$ (and have a
single descent). As $\mathbf{r}$ is ascending, it follows from the
definition that all maximal chains on $[\hat{0},x]$ must be ascending.\smallskip{}

(2.) First, suppose that for some $z\gtrdot x$ the chain $\mathbf{r}\cup\{z\}$
has an ascent at $x$. Then further extending $\mathbf{r}\cup\{z\}$
with a maximal extension of $\mathbf{a}^{\mathbf{r}\cup\{z\},[z,\hat{1]}}$
gives a chain with a single descent that is lexicographically greater
than $\mathbf{c}$, contradicting the choice of $\mathbf{c}$. In
the case where $\lambda$ is a quasi-$EL$-labeling, the result now
easily follows.

In the general quasi-$CL$ case, we claim that if some other maximal
chain $\mathbf{d}$ has an ascent at $x$, then $\mathbf{r}\cup\{z\}$
also has an ascent at $x$ (where $z\gtrdot x$ in $\mathbf{d}$).
Suppose not, and let $u<x$ be the last element of $\mathbf{c}$ such
that $\mathbf{c}$ restricted to $[\hat{0},u]$ can be extended to
an ascending chain $\mathbf{c}'$ on $[\hat{0},z]$. (We notice that
$x$ may not be in $\mathbf{c}'$.) Further let $\gamma$ and $\gamma'$
be the labels of the cover relations following $u$ in $\mathbf{c}$
and $\mathbf{c}'$ respectively. 

By Lemma \ref{lem:quasiELfirstdiff} on $[u,z]$, we have that $\gamma>\gamma'$.
As $\mathbf{c}$ and $\mathbf{c}'$ agree on $[\hat{0},u]$ and both
have an ascent at $u$, we see that the ascent in $\mathbf{c}$ at
$u$ must be strict, hence that $u\in\mathbf{a}^{[\hat{0},x]}$. By
part (\ref{enu:qELTech-TrivInterval}) we have that $\mathbf{d}$
is ascending on $[\hat{0},x]$ (and indeed on $[\hat{0},z]$). It
follows from definition that $u\in\mathbf{d}$, thus that the cover
relation following $u$ in $\mathbf{d}$ receives the same label $\gamma$
as that following $u$ in $\mathbf{c}$. 

Moreover, $\mathbf{d}$ has a strict ascent at $u$, hence $u\in\mathbf{a}^{[\hat{0},z]}$.
But then the cover relation following $u$ in $\mathbf{d}$ receives
the same label $\gamma'$ as that following $u$ in $\mathbf{c}'$.
Thus $\gamma=\gamma'$, our desired contradiction.\smallskip{}

(3.) Any such $w\lessdot x\lessdot z$ is a descent with respect to
any root, hence there is another (ascending) chain on $[w,z]$. We
take $y$ from this chain.\smallskip{}

(4.) Part (\ref{enu:qELTech-KeepsCovers}) shows that the cover relations
of $P\setminus x$ are exactly those of $P$ that do not involve $x$,
so the restriction of $\lambda$ is a chain edge labeling. Part (\ref{enu:qELTech-NoAsc})
shows that $x$ is not contained in an ascending chain on any rooted
interval, so the restriction remains a quasi-$CL$-labeling.
\end{proof}
The $x$ of Lemma \ref{lem:qELTech} will be the shedding vertex in
our vertex-decomposability proof, thus our shelling order is (perhaps
unsurprisingly) essentially lexicographic.\smallskip{}

If $\mathbf{c}$ is a maximal chain and $\alpha$ is a label, we say
that $\alpha$ is \emph{repeated} on $\mathbf{c}$ if at least two
cover relations of $\mathbf{c}$ are labeled with $\alpha$.
\begin{thm}
\label{thm:QuasiELShellsSkeleton}Let $P$ be a bounded poset with
a quasi-$CL$-labeling $\lambda$. For a maximal chain $\mathbf{c}$
let $\ell_{0}(\mathbf{c})$ denote the number of distinct labels,
and $\ell_{1}(\mathbf{c})$ denote the number of repeated labels.
If $r=\min_{\mathbf{c}}\left(\ell_{0}(\mathbf{c})+\ell_{1}(\mathbf{c})\right)$,
then $\skel_{r-2}\vert P\vert$ is vertex decomposable.\end{thm}
\begin{proof}
We first remark that the condition implies immediately that all maximal
chains contain at least $r+1$ elements, hence that $\skel_{r-2}\vert P\vert$
is pure. Thus, Lemma \ref{lem:JoinOfVDSkelsIsVD} applies. We proceed
by induction on the number of elements in $P$.\smallskip{}

\noindent \emph{Base case: }If every maximal chain of $P$ is weakly
ascending then 
\[
\vert P\vert=\vert\mathbf{a}^{[\hat{0},\hat{1}]}\vert*\link_{\vert P\vert}\vert\mathbf{a}^{[\hat{0},\hat{1}]}\vert.
\]
 We also observe that every chain in $P$ has the same set of labels
up to multiplicity, and if one chain has two or more labels on $[a_{i}^{[\hat{0},\hat{1}]},a_{i+1}^{[\hat{0},\hat{1}]}]$
then every chain does. Thus, neither $\ell_{0}$ nor $\ell_{1}$ depend
on $\mathbf{c}$. Then $\mathbf{a}^{[\hat{0},\hat{1}]}$ is a chain,
hence $\vert\mathbf{a}^{[\hat{0},\hat{1}]}\vert$ is a $(\ell_{0}-2)$-dimensional
simplex and in particular is vertex-decomposable. On the other hand,
the $(\ell_{1}-1)$-skeleton of $\link_{\vert P\vert}\vert\mathbf{a}^{[\hat{0},\hat{1}]}\vert$
is vertex-decomposable via Lemma \ref{lem:JoinOfVDSkelsIsVD}, since
$\link_{\vert P\vert}\vert\mathbf{a}^{[\hat{0},\hat{1}]}\vert$ is
the join of (the order complexes of) $\ell_{1}$ intervals, each of
which has a vertex-decomposable $0$-skeleton. A second application
of Lemma \ref{lem:JoinOfVDSkelsIsVD} gives the result. \smallskip{}

\noindent \emph{Inductive step:} If $P$ has some maximal chain with
a descent, then choose $x$ as in Lemma \ref{lem:qELTech}. Then Lemma
\ref{lem:qELTech} Part (\ref{enu:qELTech-KeepsCovers}) shows that
$x$ is a shedding vertex. Moreover $(\skel_{r-2}\vert P\vert)\setminus x=\skel_{r-2}\vert P\setminus x\vert$
is vertex-decomposable by Lemma \ref{lem:qELTech} Part (\ref{enu:qELTech-Restricts})
and induction. It remains to show that the link is vertex-decomposable.

But we have that
\[
\link_{\left(\skel_{r-2}\vert P\vert\right)}x=\skel_{r-3}\left(\link_{\vert P\vert}x\right)=\skel_{r-3}\left(\vert[\hat{0},x]\vert*\vert[x,\hat{1}]\vert\right).
\]
By Lemma \ref{lem:qELTech} Part (\ref{enu:qELTech-TrivInterval})
all maximal chains in $[\hat{0},x]$ are ascending, hence (as previously
remarked) every such chain has exactly $i=\ell_{0}^{[\hat{0},x]}$
distinct labels, and $j=\ell_{1}^{[\hat{0},x]}$ repeated labels.
But by the hypothesis, every maximal chain $\mathbf{c}$ on $[x,\hat{1}]$
must have $\ell_{0}(\mathbf{c})+\ell_{1}(\mathbf{c})\geq r-i-j$.
By induction we get that $\skel_{i+j-2}\left(\vert[\hat{0},x]\vert\right)$
and $\skel_{r-i-j-2}\left(\vert[x,\hat{1}]\vert\right)$ are each
vertex-decomposable, and then Lemma \ref{lem:JoinOfVDSkelsIsVD} gives
the desired result that $\skel_{r-3}\left(\vert[\hat{0},x]\vert*\vert[x,\hat{1}]\vert\right)$
is vertex-\-decomposable.\end{proof}
\begin{cor}
In the situation of Theorem \ref{thm:QuasiELShellsSkeleton}, $\depth\vert P\vert\geq r-2$.\end{cor}
\begin{example}
In the lattice pictured in Figure \ref{fig:quasiELexample}, $\ell_{0}$
is 2 and $\ell_{1}$ is 1, so Theorem \ref{thm:QuasiELShellsSkeleton}
tells us that the 1-skeleton is shellable and the depth is at least
1. Since the interval $[c,\hat{1}]$ is disconnected, the depth is
in fact exactly 1.
\end{example}
To prove Conjecture \ref{con:HershConj}, it then suffices to show
that all maximal chains in a modular quasi-$EL$-labeling have enough
distinct labels. We begin with a computation:
\begin{lem}
\label{lem:LeftmodLabelingsGoodness}If $\lambda$ is the left-modular
labeling with respect to left-modular chain $\mathbf{m}=\{\hat{0}=m_{0}<m_{1}<\dots<m_{r}=\hat{1}\}$
and $\lambda(x\lessdot y)=i$, then $(m_{i-1}\vee x)\wedge m_{i}<(m_{i-1}\vee y)\wedge m_{i}$.\end{lem}
\begin{proof}
We have that $x\vee m_{i-1}\wedge y=x$, hence that 
\[
\left((m_{i-1}\vee x)\wedge m_{i}\right)\wedge y=(x\vee m_{i-1})\wedge y\wedge m_{i}=x\wedge m_{i},
\]
while $\left((m_{i-1}\vee y)\wedge m_{i}\right)\wedge y=y\wedge m_{i}$
trivially. If the result is not true, then $x\wedge m_{i}=y\wedge m_{i}$,
hence 
\[
y=x\vee(m_{i}\wedge y)=x\vee(m_{i}\wedge x)=x,
\]
a contradiction.
\end{proof}
Lemma \ref{lem:LeftmodLabelingsGoodness} essentially says that the
``projection'' map $x\mapsto(m_{i-1}\vee x)\wedge m_{i}$ sends
a cover relation labeled by $i$ to distinct elements (though not
necessarily a cover relation) in the corresponding $[m_{i-1},m_{i}]$.

The following theorem then generalizes Conjecture \ref{con:HershConj}
in graded lattices:
\begin{thm}
\label{thm:GradedQLeftModularIsVD}If $\mathbf{m}=\{\hat{0}=m_{0}<m_{1}<\dots<m_{r}=\hat{1}\}$
is a left-modular chain in a graded lattice $L$, and $s$ of the
intervals $[m_{i-1},m_{i}]$ are nontrivial, then $L$ has a quasi-$EL$-labeling
assigning each maximal chain $r$ distinct labels and $s$ repeated
labels. In particular, $\skel_{r+s-2}\vert L\vert$ is vertex-decomposable.\end{thm}
\begin{proof}
We examine the left-modular quasi-$EL$-labeling: It is obvious that
$(m_{i-1}\vee x)\wedge m_{i}\leq(m_{i}\vee y)\wedge m_{i}$ for all
$x\lessdot y$, with the inequality strict if $\lambda(x\lessdot y)=i$.
A maximal chain $\mathbf{c}$ thus determines a chain in $[m_{i-1},m_{i}]$
by projecting each $x$ to $(m_{i-1}\vee x)\wedge m_{i}$. Since (as
$L$ is graded) $\mathbf{c}$ has the same length as $\bigcup_{i=0}^{r-1}[m_{i-1},m_{i}]$,
each label $i$ must occur exactly $\length[m_{i-1},m_{i}]$ times.
The final assertion follows from Theorem \ref{thm:QuasiELShellsSkeleton}.\end{proof}
\begin{rem}
Left-modular elements seem to have an especially strong impact in
a graded lattice. Another example of this is the result of McNamara
and Thomas (\cite[Theorem 1]{McNamara/Thomas:2006}, see also \cite{Thomas:2005}
for a purely lattice-theoretic proof) that a lattice is supersolvable
(graded with a maximal chain consisting of left-modular elements)
if and only if the lattice admits a certain decomposition into distributive
sublattices. 
\end{rem}
We will need the following fact about (two-sided) modular elements:
\begin{lem}
\emph{\label{lem:ModgenDistributive}(essentially in \cite{Birkhoff:1967},
extended in \cite{Stanley:1972}, see also \cite{Thomas:2005})}\\
If $\mathbf{m}=\{\hat{0}=m_{0}<m_{1}<\dots<m_{r}=\hat{1}\}$ is a
(two-sided) modular chain, then the sublattice generated by $\mathbf{m}$
and any other chain $\mathbf{c}$ is distributive.
\end{lem}
Conjecture \ref{con:HershConj} is then a consequence of the following
theorem:
\begin{thm}
\label{thm:ModRskel}If $\mathbf{m}=\{\hat{0}=m_{0}<m_{1}<\dots<m_{r}=\hat{1}\}$
is a modular chain in any lattice $L$, then $L$ has a quasi-$EL$-labeling
assigning each maximal chain $r$ distinct labels. In particular,
$\skel_{r-2}\vert L\vert$ is vertex-decomposable.\end{thm}
\begin{proof}
We examine the modular quasi-$EL$-labeling $\lambda$: Let $\mathbf{c}$
be a maximal chain. Then the sublattice $L_{0}$ generated by $\mathbf{c}$
and $\mathbf{m}$ is graded (since distributive), and moreover $\mathbf{m}$
is a modular chain in $L_{0}$. Thus $L_{0}$ has a modular labeling
$\lambda_{0}$ with respect to $\mathbf{m}$, and every chain in $L_{0}$
receives $r$ distinct labels from $\lambda_{0}$ by Theorem \ref{thm:GradedQLeftModularIsVD}.
Since $\lambda_{0}$ and $\lambda$ by definition give the same labels
to $\mathbf{c}$, every maximal chain $\mathbf{c}$ receives $r$
distinct labels, and we apply Theorem \ref{thm:QuasiELShellsSkeleton}.
\end{proof}
\medskip{}
In certain situations a quasi-$CL$-labeling will even give shellability
of the entire poset:
\begin{thm}
\label{thm:ShellableForQuasiEL} If $\lambda$ is a quasi-$CL$-labeling
on a bounded poset $P$ such that no maximal chain $\mathbf{c}$ has
more than two repeated labels in a row, then $\vert P\vert$ is vertex-decomposable.\end{thm}
\begin{proof}
Examine the proof of Theorem \ref{thm:QuasiELShellsSkeleton}. Since
the repeated label condition of our hypothesis is closed under taking
induced subposets and intervals, and the inductive step of the proof
produces a shedding vertex, we need only show that the base case is
vertex-decomposable. Then in the base case (all chains weakly ascending),
we have $\vert P\vert=\vert\mathbf{a}^{[\hat{0},\hat{1}]}\vert*\link_{\vert P\vert}\vert\mathbf{a}^{[\hat{0},\hat{1}]}\vert$,
and the repeated label condition gives that $\link_{\vert P\vert}\vert\mathbf{a}^{[\hat{0},\hat{1}]}\vert$
is exactly the join of $0$-dimensional complexes, hence vertex-decomposable.\end{proof}
\begin{rem}
It is not difficult to show under the conditions of Theorem \ref{thm:ShellableForQuasiEL}
that $\lambda$ is actually a $CC$-labeling, in the sense of Kozlov
\cite{Kozlov:1997}. 
\end{rem}
If $L$ admits a left-modular maximal chain, then the associated left-modular
labeling is an $EL$-labeling \cite{Bjorner:1980,Liu:1999}. An immediate
consequence of Theorem \ref{thm:ShellableForQuasiEL} and Lemma \ref{lem:LeftmodLabelingsGoodness}
is the following surprising result.
\begin{cor}
\label{cor:EveryOtherModIsShellable} Let $L$ be a lattice admitting
a left-modular chain $\mathbf{m}=\{\hat{0}=m_{0}<m_{1}<\dots<m_{r}=\hat{1}\}$
such that each interval $[m_{i-1},m_{i}]$ has length at most 2. (I.e.,
$L$ has a maximum length chain where at least every other element
is left-modular.) Then $L$ is vertex-decomposable, hence shellable.\end{cor}
\begin{rem}
Example 2 of \cite{Hersh/Shareshian:2006} considers the intersection
lattice of a certain modification of the braid arrangement, and makes
the claim that it is not shellable. Since the given intersection lattice
has a maximal chain with all but a single element modular, Corollary
\ref{cor:EveryOtherModIsShellable} shows this claim to be incorrect.
The main property of interest in \cite{Hersh/Shareshian:2006} was
connectivity, and the connectivity calculation is correct. I am grateful
to Hugh Thomas for pointing out to me that this lattice is indeed
shellable.
\end{rem}

\section{\label{sec:DMtechniques}Discrete Morse matchings}

A $CL$-labeling for $P$ has previously been observed \cite{Babson/Hersh:2005}
to give rise to a discrete Morse function on $\vert P\vert$. In this
section we describe similar results for quasi-$CL$-labelings. The
critical cells correspond with weakly descending maximal chains, giving
an approach to computing the homotopy type that extends that of Björner
and Wachs for a $CL$-shellable poset.

\subsection{\label{sub:ReviewOfDiscreteMorse}Review of discrete Morse theory}

Discrete Morse theory was developed by Forman \cite{Forman:1998},
although the essential matching idea was earlier discovered by Brown
\cite{Brown:1992}. In discrete Morse theory, one constructs a partial
matching between faces of adjacent dimensions in a simplicial complex
$\Delta$. The matched faces can then be collapsed, leaving a $CW$-complex
$X$ homotopic to $\Delta$, and with cells in one-to-one correspondence
with the unmatched faces (or \emph{critical cells}) of $\Delta$. 

Babson and Hersh \cite{Babson/Hersh:2005} showed how to create a
discrete Morse matching from the lexicographic ordering induced by
an edge labeling on all maximal chains of $P$. The topological consequences
of a $CL$-labeling are recovered as a special case. We briefly summarize
this work, and in Section \ref{sub:DiscreteMorse-quasiCL} apply it
to quasi-$CL$-labelings. 

Let $\lambda$ be any edge labeling (or chain-edge labeling) of a
bounded poset $P$. Lexicographically order the maximal chains of
$P$ according to $\lambda$, breaking ties consistently, for example
by taking a linear extension $\epsilon$ of $P$ and extending $\lambda$
to $\lambda_{+}(x\lessdot y)=(\lambda(x\lessdot y),\epsilon(y))$.
A \emph{skipped interval} of a maximal chain $\mathbf{c}=\{\hat{0}=c_{0}\lessdot c_{1}\lessdot\dots\lessdot c_{\ell}=\hat{1}\}$
is a pair $c_{i}\leq c_{j}$ such that $\mathbf{c}\setminus[c_{i},c_{j}]$
is contained in some maximal chain $\mathbf{c}'$ with $\mathbf{c}'\lexlt\mathbf{c}$,
or (degenerately) the pair $c_{0}<c_{\ell}$ for the lexicographically
first maximal chain. A \emph{minimal skipped interval} is a skipped
interval which is minimal under inclusion. We notice that a subchain
$\mathbf{d}\subset\mathbf{c}$ fails to be contained in an earlier
$\mathbf{c}'$ if and only if $\mathbf{d}$ contains some $c_{k}$
with $c_{i}\leq c_{k}\leq c_{j}$ for each minimal skipped interval
$c_{i}\leq c_{j}$. Thus, the `new' faces in $\mathbf{c}$ are exactly
those that contain a vertex in each minimal skipped interval.

For each maximal chain $\mathbf{c}$, we ``shrink'' the minimal
skipped intervals of $\mathbf{c}$ by a certain sequence of truncating
and discarding operations to obtain a set of intervals $\mathcal{J}(\mathbf{c})$.
The details of how $\mathcal{J}(\mathbf{c})$ is obtained will not
be important to us, except that the intervals in $\mathcal{J}(\mathbf{c})$
do not overlap, that each interval in $\mathcal{J}(\mathbf{c})$ is
contained in a minimal skipped interval, and that if a minimal skipped
interval has length $0$ (i.e. $c_{i}=c_{j}$), then it is preserved
in passing to $\mathcal{J}(\mathbf{c})$. 

The main theorem of poset Morse theory is then:
\begin{thm}
\emph{\label{thm:PosetMorseMatching}(Babson and Hersh }\cite[Theorem 2.2]{Babson/Hersh:2005}\emph{)}\\
Let $P$ be a bounded poset, $\lambda$ be a chain-edge labeling,
and $\mathcal{J}(\mathbf{c})$ be as described above. Then there is
a Morse matching such that for any maximal chain $\mathbf{c}$ 
\begin{enumerate}
\item $\mathbf{c}$ contains at most one critical cell.
\item $\mathbf{c}$ contains a critical cell if and only if $\mathcal{J}(\mathbf{c})$
covers $\mathbf{c}\setminus\{\hat{0},\hat{1}\}$.
\item In the case where $\mathcal{J}(\mathbf{c})$ covers $\mathbf{c}\setminus\{\hat{0},\hat{1}\}$,
the unique critical cell in $\mathbf{c}$ has dimension $\#\mathcal{J}(\mathbf{c})-1$.
\end{enumerate}
\end{thm}
An easy lower bound for the dimension of the critical cell associated
with $\mathbf{c}$ is the number of minimal skipped intervals of length
0 for $\mathbf{c}$, as minimal skipped intervals of length 0 are
preserved in $\mathcal{J}(\mathbf{c})$. An improved lower bound is
the number of minimal skipped intervals of length 0, plus the number
of nonempty connected components left in the Hasse diagram for $\mathbf{c}$
after deleting the minimal skipped intervals of length 0.

For more details, we refer the reader to the original paper of Babson
and Hersh \cite{Babson/Hersh:2005}, to the helpful follow-up paper
\cite{Hersh:2005}, and to the highly readable overview in \cite{Sagan/Vatter:2006}.

\subsection{\label{sub:DiscreteMorse-quasiCL}Discrete Morse matchings for quasi-$CL$-labelings}

We consider the minimal skipped intervals in the lexicographic order
induced by a quasi-$CL$-labeling.
\begin{lem}
\label{lem:quasiELMorse}Let $P$ be a bounded poset with a quasi-$CL$-labeling
$\lambda$, and let $\mathbf{c}=\{\hat{0}=c_{0}\lessdot c_{1}\lessdot\dots\lessdot c_{\ell}=\hat{1}\}$
be a maximal chain of $P$. Then 
\begin{enumerate}
\item If $\mathbf{c}$ has a strict descent at $c_{k}$, then $\{c_{k}\}$
is a minimal skipped interval for $\mathbf{c}$.
\item If $\mathbf{c}$ has a strict ascent at $c_{k}$, then $c_{k}$ is
not contained in any minimal skipped interval for $\mathbf{c}$ unless
$\mathbf{c}$ is (the degenerate case of) the lexicographically first
maximal chain.
\end{enumerate}
\end{lem}
\begin{proof}
(1.) The chain $\mathbf{c}'$ obtained by replacing the descent with
an ascent is lexicographically earlier, and $\mathbf{c}'\cap\mathbf{c}=\mathbf{c}\setminus\{c_{k}\}$.

(2.) Suppose by contradiction that $\mathbf{c}$ has a strict ascent
at $c_{k}$ and that $c_{i}\leq c_{j}$ is a minimal skipped interval
with $i\leq k\leq j$. If $\mathbf{c}$ has any descent in $[c_{i-1},c_{j+1}]$,
then part (1) gives a smaller skipped interval, contradicting minimality
of the skipped interval. Thus $\mathbf{c}$ is (weakly) ascending
on the interval $[c_{i-1},c_{j+1}]$. Let $\mathbf{c}'$ be the lexicographically
minimal preceding chain with $\mathbf{c}\setminus[c_{i},c_{j}]\subseteq\mathbf{c}'$.
Then by definition of quasi-$CL$-labeling $\mathbf{c}'$ must also
be weakly ascending on $[c_{i-1},c_{j+1}]$, hence contain $c_{k}$. 

If $i=j=k$ then $c_{k-1}\lessdot c_{k}\lessdot c_{k+1}$ is strictly
increasing, and uniqueness of $\mathbf{a}^{[c_{k-1},c_{k+1}]}$ gives
a contradiction. Otherwise, $\mathbf{c}'$ restricted to $[c_{i-1},c_{k}]$
or $[c_{k},c_{j+1}]$ is $\lexleq$ the restriction of $\mathbf{c}$
to the same interval, and the inequality is strict for at least one
such restriction. It follows that either $\mathbf{c}\setminus[c_{i},c_{k-1}]$
or $\mathbf{c}\setminus[c_{k+1},c_{j}]$ is contained in a lexicographically
earlier chain, contradicting minimality of the skipped interval $c_{i}\leq c_{j}$.
\end{proof}
Lemma \ref{lem:quasiELMorse} characterizes the resulting Morse matching:
\begin{thm}
\label{thm:qCLhomotopyFromDescs}If $P$ is a bounded poset with a
quasi-$CL$-labeling $\lambda$, then $\vert P\vert$ has a poset
Morse matching such that a maximal chain $\mathbf{c}$ contributes
a critical cell only if $\mathbf{c}$ is weakly descending. If $\ell_{0}(\mathbf{c})$
is the number of distinct labels and $\ell_{1}(\mathbf{c})$ the number
of repeated labels of a maximal chain $\mathbf{c}$, , then the dimension
of the cell associated to a weakly descending chain $\mathbf{c}$
is at least $\ell_{0}(\mathbf{c})+\ell_{1}(\mathbf{c})-2$.\end{thm}
\begin{proof}
Lemma \ref{lem:quasiELMorse} part (2) tells us that if a chain has
any strict ascent, then $c_{k}$ is not covered by $\mathcal{J}(\mathbf{c})$.
Conversely, Lemma \ref{lem:quasiELMorse} part (1) tells us that each
strict descent is a minimal skipped interval of length 0. 

We observe that $\ell_{0}(\mathbf{c})-1$ is the number of strict
descents, and $\ell_{1}(\mathbf{c})$ is the number of nonempty components
remaining in the Hasse diagram of $\mathbf{c}$ after deleting the
strict descents. The dimension bound then follows from Theorem \ref{thm:PosetMorseMatching}
and the discussion following its statement.
\end{proof}
Theorem \ref{thm:qCLhomotopyFromDescs} is an extension of \cite[Theorem 5.9]{Bjorner/Wachs:1996}
to quasi-$CL$-labelings, following the approach of \cite[Proposition 4.1]{Babson/Hersh:2005}.
\begin{cor}
\label{cor:AlternatingConnectivity}Let $P$ be a poset with a quasi-$CL$-labeling
$\lambda$, and for a maximal chain $\mathbf{c}$ let $\ell_{0}(\mathbf{c})$
and $\ell_{1}(\mathbf{c})$ be as in Theorem \ref{thm:qCLhomotopyFromDescs}.
Then the connectivity of $P$ is at least 
\[
\min\left\{ \ell_{0}(\mathbf{c})+\ell_{1}(\mathbf{c})-3\,:\,\mathbf{c}\mbox{ a weakly descending maximal chain}\right\} .
\]

\end{cor}
We notice that Corollary \ref{cor:AlternatingConnectivity} requires
examination of only weakly descending chains of a quasi-$CL$-labeling.
This is in contrast to Theorem \ref{thm:ShellableForQuasiEL}, which
requires examining the label sets of all chains, although of course
Theorem \ref{thm:ShellableForQuasiEL} has the stronger consequence
of shellability.

We further remark that the approach of Theorem \ref{thm:qCLhomotopyFromDescs}
and Corollary \ref{cor:AlternatingConnectivity} reduces understanding
the homotopy type of a poset with a quasi-$CL$-labeling to understanding
the intervals between descents on the weakly descending chains. 

\smallskip{}

We now apply Corollary \ref{cor:AlternatingConnectivity} to left-modular
labelings. By a \emph{chain of complements} to a left-modular chain
$\mathbf{m}$, we mean a chain consisting of a complement to every
element of $\mathbf{m}$. We notice it is an immediate consequence
of the definition that no left-modular element may have two comparable
complements, so that a chain of complements cannot be longer than
$\mathbf{m}$. On the other hand, two comparable left-modular elements
may have the same complement, so a chain of complements may be shorter
than $\mathbf{m}$. In the two-sided modular case, Lemma \ref{lem:ModgenDistributive}
gives that any chain of complements has exactly the same length as
$\mathbf{m}$.

The following lemma then extends \cite[Lemma 1.2]{Thevenaz:1985}.
\begin{lem}
\label{lem:DescChainsModularLabeling}If $\lambda$ is the left-modular
quasi-$EL$-labeling of $L$ with respect to $\mathbf{m}$, then a
maximal chain $\mathbf{c}$ is weakly descending if and only if $\mathbf{c}$
is a refinement of a chain of complements to $\mathbf{m}$.\end{lem}
\begin{proof}
The ``if'' direction is a straightforward computation: if $y_{\ell}$
and $y_{\ell-1}$ are complements to $m_{\ell}$ and $m_{\ell-1}$
with $y_{\ell}<y_{\ell-1}$, then every cover relation on $[y_{\ell},y_{\ell-1}]$
receives label $\ell$. 

For the other direction, we let $\mathbf{c}=\{\hat{0}=c_{0}<c_{1}<\dots<c_{k}=\hat{1}\}$
be a weakly descending chain, with $j$ the smallest index such that
$\lambda(c_{j}\lessdot c_{j+1})\leq\ell$. We notice that if $\lambda(c_{i}\lessdot c_{i+1})>\ell$,
then $m_{\ell}\wedge c_{i+1}\leq c_{i}$, hence $m_{\ell}\wedge c_{i+1}=m_{\ell}\wedge c_{i}$.
Conversely, if $\lambda(c_{i}\lessdot c_{i+1})\leq\ell$, then $m_{\ell}\vee c_{i}\geq c_{i+1}$,
hence $m_{\ell}\vee c_{i}=m_{\ell}\vee c_{i+1}$. Applying these observations
inductively, we see that $c_{j}\wedge m_{\ell}=\hat{0}\wedge m_{\ell}=\hat{0}$,
while $c_{j}\vee m_{\ell}=\hat{1}\vee m_{\ell}=\hat{1}$, so that
$c_{j}$ is a complement to $m_{\ell}$. Hence $\mathbf{c}$ contains
a complement to each $m_{\ell}\in\mathbf{m}$.\end{proof}
\begin{cor}
\label{cor:ChainOfComplementsConnectivity}If $\mathbf{m}$ is a left-modular
chain in a lattice $L$, then $\vert L\vert$ is $(s+t-3)$-connected,
where $s$ and $t$ are the smallest number of distinct and repeated
labels (respectively) in a maximal refinement of a chain of complements
to $\mathbf{m}$.
\end{cor}
We began the paper by recalling the result of Hersh and Shareshian
that if a lattice $L$ admits a modular chain $\mathbf{m}$ of length
$r$, then $\vert L\vert$ is $(r-3)$-connected \cite[Theorem 1]{Hersh/Shareshian:2006}.
We observed in Theorem \ref{thm:ModRskel} that every chain in such
a lattice receives $r$ distinct labels from the modular labeling.
The shellability consequence of Theorem \ref{thm:ModRskel} gives
one new generalization of \cite[Theorem 1]{Hersh/Shareshian:2006};
Corollary \ref{cor:ChainOfComplementsConnectivity} gives another.

\section{\label{sec:SubgroupLatApplications}Applications to the subgroup
lattice}

For a group $G$, let $L(G)$ denote the \emph{subgroup lattice} of
$G$, that is, the lattice consisting of all subgroups of $G$ ordered
by inclusion. The meet and join operations in this lattice are $H\vee K=\langle H,K\rangle$,
and $H\wedge K=H\cap K$. The Dedekind identity from group theory
gives us that any normal subgroup is modular in $L(G)$. Series of
normal subgroups form an important class of examples of modular chains. 

The topology of $\vert L(G)\vert$ has been especially studied in
the solvable case, where one has long chains of modular elements.
Thévenaz \cite{Thevenaz:1985} showed:
\begin{thm}
\label{thm:Thevenaz}\emph{(Thévenaz \cite[Theorem 1.4]{Thevenaz:1985})}
If $G$ is solvable with a chief series of length $r$ then $L(G)$
has the homotopy type of a bouquet of $(r-2)$-dimensional spheres,
where the spheres are in bijective correspondence with the chains
of complements to the chief series. 
\end{thm}
We recall that the geometry of $\vert L(G)\vert$ can in fact be used
to classify solvable groups:
\begin{thm}
\label{thm:SolvableIffShellable}For a finite group $G$, TFAE:
\begin{enumerate}
\item $G$ is solvable.
\item $L(G)$ is shellable. \cite{Shareshian:2001}
\item $L(G)$ has an $EL$-labeling. \cite{Woodroofe:2008}
\end{enumerate}
\end{thm}
The proof of Theorem \ref{thm:SolvableIffShellable} proceeds roughly
as follows. The direction $(3)\implies(2)$ is immediate. To show
$(1)\implies(3)$, refine the modular labeling for a solvable group
into an $EL$-labeling \cite[Theorem 4.1]{Woodroofe:2008}. For $(2)\implies(1)$,
Shareshian applies the classification of minimal simple groups, and
calculates enough information about the homotopy type for such a group
$G$ to show $L(G)$ is not shellable \cite[Section 3]{Shareshian:2001}.
\medskip{}

One feature of the $EL$-labeling from \cite{Woodroofe:2008} is that
the descending chains are exactly the chains of complements to the
chief series, giving a new proof of Theorem \ref{thm:Thevenaz}. We
observe that Theorem \ref{thm:Thevenaz} also follows from the modular
quasi-$EL$-labeling and Theorem \ref{thm:qCLhomotopyFromDescs},
and for essentially the same reasons.

While a new framework for understanding Theorem \ref{thm:Thevenaz}
has some appeal, the topology of $\vert L(G)\vert$ for a solvable
group is already well-understood. The real advantage of studying the
quasi-$EL$-labeling on $L(G)$ is that it is applicable to non-solvable
groups. We use this to give the new characterization of solvability
stated in Theorem \ref{thm:Intro-SolvableShellSkel} and Corollary
\ref{cor:Intro-SolvableDepth}:
\begin{proof}
[Proof of Theorem \ref{thm:Intro-SolvableShellSkel}/Corollary \ref{cor:Intro-SolvableDepth}]If
$G$ is a solvable group, then $L(G)$ is shellable by Theorem \ref{thm:SolvableIffShellable}.
Kohler \cite{Kohler:1968} proves the minimum length of a maximal
chain in the subgroup lattice of a solvable group to be $r$, hence
the minimum facet dimension and depth of $\vert L(G)\vert$ are $r-2$. 

Conversely, if $G$ is not solvable then all maximal chains have length
at least $r+2$ \cite[Theorem 1.4]{Shareshian/Woodroofe:2012}. By
Lemma \ref{lem:ModgenDistributive} each maximal chain has $r$ distinct
labels with respect to the modular labeling, and a pigeonhole argument
shows that in any maximal chain some label is repeated. Theorem \ref{thm:QuasiELShellsSkeleton}
then gives that $\depth\vert L(G)\vert\geq r-1$.
\end{proof}
These results at first glance appear somewhat surprising. One usually
considers shellability to be a tool to show that a simplicial complex
has strong properties related to Cohen-Macaulay, but in this situation
it is the non-shellable complexes which have higher depth (relative
to $r$). Theorem \ref{thm:Intro-SolvableShellSkel} and Corollary
\ref{cor:Intro-SolvableDepth} are essentially a consequence of non-solvable
groups having longer maximal chains than might be expected solely
from their modular structure.

If one restricts oneself to groups where $L(G)$ is not contractible,
then a similar characterization holds for connectivity by Corollary
\ref{cor:AlternatingConnectivity} and an argument parallel to that
of Theorem \ref{thm:Intro-SolvableShellSkel}:
\begin{cor}
Let $G$ be a finite group with a chief series of length $r$. If
$L(G)$ is $(r-2)$-connected, then either $L(G)$ is contractible
or else $G$ is not solvable.
\end{cor}
\bibliographystyle{hamsplain}
\bibliography{3_Users_russw_Documents_Research_Master}

\end{document}